\documentclass[12pt]{article}

\usepackage{amssymb}
\usepackage{amsmath}
\usepackage{amsthm}
\usepackage{mathtools}
\usepackage[hmargin=2.5cm,vmargin=2.5cm]{geometry}
\usepackage{color}

\author{N.B-S. Boer$^*$ \and A.E. Sterk\footnote{Bernoulli Institute for Mathematics, Computer Science, and Artificial Intelligence, 
University of Groningen, PO Box 407, 9700 AK Groningen, The Netherlands. E-mail: \texttt{n.b.boer@student.rug.nl}, \texttt{a.e.sterk@rug.nl}. Corresponding author: A.E.~Sterk.}}
\title{Generalized Fibonacci numbers and extreme value laws for the R\'enyi map}
\date{\today}

\newtheorem{theorem}{Theorem}[section]
\newtheorem{lemma}[theorem]{Lemma}

\newtheorem{definition}[theorem]{Definition}

\DeclareMathOperator{\RE}{Re}
\DeclareMathOperator{\prob}{\mathbb{P}}
\DeclareMathOperator{\mean}{\mathbb{E}}
\DeclareMathOperator{\Leb}{Leb}

\begin{document}

\maketitle

\begin{abstract}
In this paper we prove an extreme value law for a stochastic process obtained by iterating the R\'enyi map $x \mapsto \beta x \pmod 1$, where we assume that $\beta>1$ is an integer. Haiman (2018) derived a recursion formula for the Lebesgue measure of threshold exceedance sets. We show how this recursion formula is related to a rescaled version of the $k$-generalized Fibonacci sequence. For the latter sequence we derive a Binet formula which leads to a closed-form expression for the distribution of partial maxima of the stochastic process. The proof of the extreme value law is completed by deriving sharp bounds for the dominant root of the characteristic polynomial associated with the Fibonacci sequence.
\end{abstract}

\tableofcontents

\newpage

%===============================================================================

\section{Introduction}

Extreme value theory for a sequence of i.i.d.\ random variables $(X_i)_{i=0}^\infty$ studies the asymptotic distribution of the partial maximum
\begin{equation}\label{partmax}
M_n = \max(X_0,\dots,X_{n-1})
\end{equation}
as $n\to\infty$. Since the distribution of $M_n$ has a degenerate limit it is necessary to consider a rescaling. Under appropriate conditions there exist sequences $a_n > 0$ and $b_n\in\mathbb{R}$ for which the limiting distribution of $a_n(M_n-b_n)$ is nondegenerate. As an elementary example, assume that the variables $X_i\sim U(0,1)$ are independent. Then with $a_n=n$ and $b_n=1$ it follows for $\lambda \geq 0$ that
\begin{equation}\label{uniform_iid}
\lim_{n\to\infty} \prob(a_n(M_n-b_n) \leq -\lambda)
 = \lim_{n\to\infty} \prob\bigg(M_n \leq 1 - \frac{\lambda}{n}\bigg)
 = \lim_{n\to\infty}\bigg(1-\frac{\lambda}{n}\bigg)^n
 = e^{-\lambda}.
\end{equation}
More generally, it can be proven that extreme value distributions for i.i.d.\ random variables are either a Weibull, Gumbel, or Fr\'echet distribution \cite{Galambos, HF:2006, Res87}. For extensions of extreme value theory to dependent random variables, see \cite{Leadbetter}.

In the last twenty years the applicability of extreme value theory has been extended to the setting of deterministic dynamical systems. The pioneering work \cite{Collet:01} introduced many ideas that were used in subsequent papers by various authors. A particularly important development was proving the link between hitting and return time statistics on the one hand and extreme value laws on the other hand \cite{FFT:10}. Hence, extreme value laws can be proven by using the many results on hitting and return time statistics that are available. The latter have been derived for  general classes of dynamical systems \cite{Abadi:04, HLV:05, HV:09, Hirata:93, HSV:99, Rousseau:14} and go beyond the context of the piecewise linear maps that will be considered in the present paper. For a detailed account on the subject of extremes in dynamical systems the interested reader is referred to the recent monograph \cite{Lucarini2016} and the extensive list of references therein.

In this paper we consider the R\'enyi map \cite{Renyi:57} given by
\[
f : [0,1) \to [0,1), \quad f(x) = \beta x \pmod 1,
\]
where we restrict to the case where $\beta>1$ is an integer. This map is an active topic of study within the field of dynamical systems and ergodic theory. In the special case $\beta=2$ the map $f$ is also known as the doubling map which is an archetypical example of a chaotic dynamical system \cite{BroerTakens:11}. Other applications, in which also non-integer values of $\beta$ are considered, include the study of random number generators~\cite{Addabbo-et-al:07} and dynamical systems with holes in their state space \cite{KKLL:2019}.

The assumption that $\beta > 1$ is an integer implies that the Lebesgue measure is an invariant probability measure of the map $f$:
\begin{lemma}\label{id}
If $X$ is a random variable such that $X \sim U(0,1)$, then $f(X) \sim U(0,1)$.
\end{lemma}

\begin{proof}
For $u \in (0,1]$ we have that $\prob(X \in [0,u)) = u$. This gives
\[
\prob(f(X) \in [0,u))
 = \prob(X \in f^{-1}([0,u))
 = \sum_{k=1}^\beta \prob\bigg(X \in \bigg[\frac{k-1}{\beta},\frac{k-1+u}{\beta}\bigg)\bigg) = u,
\]
which implies that $f(X) \sim U(0,1)$.
\end{proof}

Consider the stochastic process $(X_i)_{i=0}^\infty$ defined by $X_{i+1} = f(X_i)$, where $X_0 \sim U(0,1)$. Lemma~\ref{id} implies that the variables $X_i$ are identically distributed, but they are no longer independent. Let $M_n$ be the partial maximum as defined in \eqref{partmax}. Haiman~\cite{Haiman:2018} proved the following result:
\begin{theorem}\label{main_result}
For fixed $\lambda>0$ and the sequence $n_k = \lfloor \beta^k\lambda\rfloor$ it follows that
\[
\lim_{k\to\infty} \prob(M_{n_k} \leq 1 - \beta^{-k}) = e^{-\frac{\beta-1}{\beta}\lambda}.
\]
\end{theorem}

\noindent
Note that for $\lambda \in \mathbb{N}$ we have $\prob(M_{n_k} \leq 1 - \beta^{-k}) = \prob(\beta^k \lambda(M_{\beta^k\lambda}-1) \leq -\lambda)$. Therefore, the result of Theorem \ref{main_result} is in spirit similar to the example in \eqref{uniform_iid}, albeit that a subsequence of $M_n$ is considered.

The aim of this paper is to give an alternative proof for Theorem \ref{main_result} which relies on asymptotic properties of a rescaled version of the $k$-generalized Fibonacci numbers. The restriction that $\beta$ is an integer is essential for our proof. Indeed, for non-integer values of $\beta>1$ the invariant measure of the map $f$ is generally different from the Lebesgue measure; see \cite{Choe:05, Renyi:57} for the case $\beta = (\sqrt{5}+1)/2$. A more general approach to establish an extreme value law would be to show that two mixing conditions are satisfied which guarantee that an extreme value law for a time series generated by a dynamical system can be obtained as if it were an i.i.d.\ stochastic process. An application of this approach to the tent map process can be found in~\cite{Freitas:09}. However, in Appendix~\ref{sec:mixing} we show that one of these conditions does not hold the R\'enyi map process.

The fact that the limit in Theorem \ref{main_result} is not equal to $e^{-\lambda}$ has a particular statistical interpretation. The coefficient $\theta := (\beta-1)/\beta$ in the exponential is called the extremal index and measures the degree of clustering in extremes arising as a consequence of the dependence between the variables $X_i$; see \cite{Leadbetter, Lucarini2016} for more details. In Appendix~\ref{sec:clustering} we show how the extremal index for the R\'enyi map process can be derived in an elementary way. For more general dynamical systems, conditions for extreme value laws with particular extremal indices are derived in \cite{FFT:12}.

%===============================================================================

\section{The relation with generalized Fibonacci numbers}

In this section we fix the numbers $k \in \mathbb{N}$ and $u = \beta^{-k}$. For any integer $i \geq 0$ we define the set
\[
E_i = \{ x \in [0,1) \,:\, f^i(x) > 1-u \},
\]
where the dependence on $k$ is suppressed in the notation for convenience. Then
\[
\prob(M_n \leq 1-u) = 1 - B_n
\quad\text{where}\quad
B_n = \Leb\bigg(\bigcup_{i=0}^{n-1} E_i\bigg),
\]
where $\Leb$ denotes the Lebesgue measure. Based on self-similarity arguments Haiman~\cite{Haiman:2018} derived the following recursion formula which holds for each fixed $k \in \mathbb{N}$:
\begin{align}
B_n		& = (n-1)\frac{\beta-1}{\beta}u + u		& \text{if } & 1 \leq n \leq k+1,		\label{regular} \\
B_{n+1}	& = B_n + \frac{\beta-1}{\beta}u(1-B_{n-k})	& \text{if } & n \geq k+1.		\label{recursion}
\end{align}
The same idea was used earlier by Haiman to study extreme value laws for the tent map \cite{Haiman:03}.

For $n \in \mathbb{Z}$ we define the following numbers:
\begin{equation}
\label{fseq}
F_n = 
\begin{cases}
0 & \text{if } n < 1, \\
1 & \text{if } n = 1, \\
\displaystyle\frac{B_n-B_{n-1}}{u/\beta^{n-1}} & \text{if } n > 1.
\end{cases}
\end{equation}
These numbers have the following geometric meaning. Note that the sets $E_i$ can be written as a union of $\beta^i$ intervals:
\[
E_i = \bigcup_{j=1}^{\beta^i} \bigg[\frac{j-u}{\beta^i}, \frac{j}{\beta^i}\bigg), \quad i \geq 0.
\]
For $n \geq 2$ the number $F_n$ equals the number of subintervals of the set $E_{n-1}$ which need to be added to $E_0 \cup \dots \cup E_{n-2}$ in order to obtain $E_0 \cup \dots \cup E_{n-1}$. Figure \ref{fig_intervals} illustrates this for the special case $\beta=2$ and $k=2$.

\begin{figure}
\includegraphics[width=\textwidth]{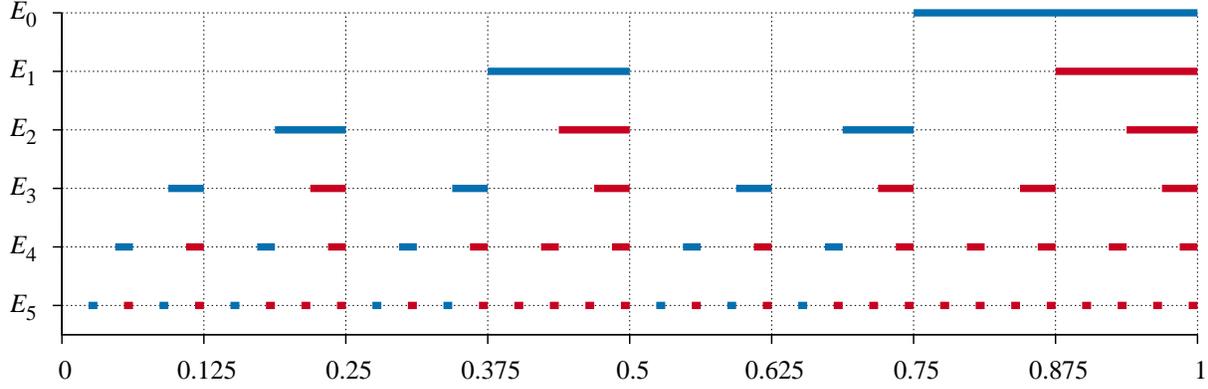}
\caption{Illustration of the sets $E_0,\dots,E_5$ for $\beta=2$ and $k=2$. Each set $E_n$ is a union of $\beta^n$ intervals. Intervals in $E_n$ which are disjoint from (resp.\ contained in) the intervals comprising $E_0,\dots,E_{n-1}$ are drawn in blue (resp.\ red). For $n \geq 2$ the number $F_n$ equals the number of subintervals of the set $E_{n-1}$ which need to be added to $E_0 \cup \dots \cup E_{n-2}$ in order to obtain $E_0 \cup \dots \cup E_{n-1}$. The figure clearly shows that $F_2=1$, $F_3=2$, $F_4=3$, $F_5=5$, and $F_6=8$ which are the starting numbers of the Fibonacci sequence.}
\label{fig_intervals}
\end{figure}

\begin{lemma}\label{prob_fib}
For any $k,n \in \mathbb{N}$ it follows that
\[
\prob (M_n \leq 1 - \beta^{-k}) = \frac{\beta^{1-n-k}}{\beta-1} F_{n+k+1}.
\]
\end{lemma}

\begin{proof}
For $n \geq k+2$ equation \eqref{recursion} gives
\[
F_n = \frac{B_n - B_{n-1}}{u/\beta^{n-1}} = (\beta-1)\beta^{n-2}(1-B_{n-k-1}),
\]
or, equivalently,
\[
B_{n-k-1} = 1 - \frac{\beta^{2-n}}{\beta-1}F_n.
\]
The proof is completed by substituting $n$ for $n-k-1$.
\end{proof}

The following result provides the connection between the sequence $(B_n)$ and generalizations of the Fibonacci numbers. In particular, for $\beta=2$ the sequence $(F_n)$ is the well-known $k$-generalized Fibonacci sequence.

\begin{lemma}\label{fib}
The following statements are equivalent:
\begin{enumerate}
\item[(i)]  Equations \eqref{regular} and \eqref{recursion} hold;
\item[(ii)] For fixed $k \in \mathbb{N}$, the sequence $(F_n)$, where $n \in \mathbb{Z}$, defined in \eqref{fseq} satisfies
\begin{equation}\label{genfib}
F_n = 
\begin{cases}
0 & \text{if } n < 1, \\
1 & \text{if } n = 1, \\
(\beta-1)(F_{n-1} + F_{n-2} + \dots + F_{n-k})  & \text{if } n \geq 2.
\end{cases}
\end{equation}
In particular, $F_n = (\beta-1)\beta^{n-2}$ for $2 \leq n \leq k+1$.
\end{enumerate}
\end{lemma}

\begin{proof}
Assume that statement (i) holds. By definition $F_1 = 1$ and for $2 \leq n \leq k+1$ equation \eqref{regular} implies that
\[
F_n = \frac{B_n - B_{n-1}}{u/\beta^{n-1}} = \frac{\beta^{n-1}}{u}\bigg[ \bigg((n-1)\frac{\beta-1}{\beta}u + u\bigg) - \bigg((n-2)\frac{\beta-1}{\beta}u+u\bigg) \bigg] = (\beta-1)\beta^{n-2}.
\]

We proceed with induction on $n$. For any $n \geq k+1$ equation \eqref{recursion} gives
\begin{equation}\label{Fnp1}
F_{n+1} = \frac{B_{n+1} - B_n}{u/\beta^n} = (\beta-1)\beta^{n-1}(1-B_{n-k}).
\end{equation}
In particular, for $n=k+1$ we have
\[
\begin{split}
F_{k+2} 
 & = (\beta-1)\beta^{k}(1-B_{1}) \\
 & = (\beta-1)(\beta^k - 1) = (\beta-1)^2\sum_{i=1}^k \beta^{k-i} = (\beta-1)\sum_{i=1}^k F_{k+2-i}.
\end{split}
\]
Assume that for some $n\geq k+1$ it follows that
\[
F_{n+1} = (\beta-1)\sum_{i=1}^k F_{n+1-i}.
\]
First using equation \eqref{fseq} and then equation \eqref{recursion} twice gives
\[
\begin{split}
F_{n+2}
 & = (\beta-1)\beta^n(1-B_{n-k+1}) \\
 & = (\beta-1)\beta^n(1-B_{n-k}) - (\beta-1)^2\beta^{n-k-1}(1-B_{n-2k}) \\
 & = \beta F_{n+1} - (\beta-1)F_{n-k+1},
\end{split}
\]
where the last equality follows from \eqref{Fnp1}. Finally, the induction hypothesis implies that
\[
\begin{split}
F_{n+2}
 & = (\beta-1)F_{n+1} + F_{n+1} - (\beta-1)F_{n-k+1} \\
 & = (\beta-1)F_{n+1} + (\beta-1)\sum_{i=1}^k F_{n+1-i} - (\beta-1)F_{n-k+1} \\
 & = (\beta-1)\sum_{i=1}^k F_{n+2-i}.
\end{split}
\]
Hence, statement (ii) follows.

Conversely, assume that statement (ii) holds. In particular, $F_n = (\beta-1)\beta^{n-2}$ for $2 \leq n \leq k+1$ so that by equation \eqref{fseq} it follows that
\[
B_n = B_{n-1} + \frac{u}{\beta^{n-1}}F_n = B_{n-1} + \frac{\beta-1}{\beta}u.
\]
Equation \eqref{regular} now follows by recalling that $B_1 = u$.

We proceed by strong induction on $n$. We have
\[
F_{k+2} = (\beta-1)\sum_{i=1}^k F_{k+2-i} = (\beta-1)\sum_{i=1}^k (\beta-1)\beta^{k-i} = (\beta-1)(\beta^{k}-1).
\]
Recalling that $B_1 = u = \beta^{-k}$, equation \eqref{fseq} implies that
\[
B_{k+2} = B_{k+1} + \frac{u}{\beta^{k+1}}F_{k+2} = B_{k+1} + \frac{u}{\beta^{k+1}}(\beta-1)(\beta^k-1) = B_{k+1} + \frac{\beta-1}{\beta}u(1 - B_1),
\]
which shows that equation \eqref{recursion} holds for $n=k+1$. Assume that there exists $m\in\mathbb{N}$ such that \eqref{recursion} holds for all $k+1 \leq n \leq m$. Observe that
\[
\begin{split}
F_{m+2}
 & = (\beta-1)\sum_{i=1}^k F_{m+2-i} \\
 & = (\beta-1)\bigg(F_{m+1} - F_{m+1-k} + \sum_{i=1}^k F_{m+1-i}\bigg) \\
 & = (\beta-1)\bigg(F_{m+1} - F_{m+1-k} + \frac{F_{m+1}}{\beta-1}\bigg) \\
 & = \beta F_{m+1} - (\beta-1)F_{m+1-k}.
\end{split}
\]
Therefore,
\[
\begin{split}
B_{m+2} - B_{m+1}
 & = \frac{u}{\beta^{m+1}}F_{m+2} \\
 & = \frac{u}{\beta^{m+1}}(\beta F_{m+1} - (\beta-1)F_{m+1-k}) \\
 & = B_{m+1} - B_m - \frac{\beta-1}{\beta^{k+1}}(B_{m+1-k}-B_{m-k}).
\end{split}
\]
The induction hypothesis gives
\[
\begin{split}
B_{m+2}-B_{m+1}
 & = \frac{\beta-1}{\beta}u(1-B_{m-k}) - \frac{(\beta-1)^2}{\beta^{k+2}}u(1-B_{m-2k}) \\
 & = \frac{\beta-1}{\beta}u\bigg(1-\bigg(B_{m-k} + \frac{\beta-1}{\beta}u(1-B_{m-2k})\bigg)\bigg) \\
 & = \frac{\beta-1}{\beta}u(1-B_{m-k+1}).
\end{split}
\]
Hence, statement (i) follows.
\end{proof}

%===============================================================================

\section{The Binet formula}

Let the sequence $(F_n)$ be as defined in \eqref{genfib}, where $\beta\geq 2$ is assumed to be an integer. In this section we will derive a closed-form expression for $F_n$ as a function of $n$ along the lines of Spickerman and Joyner~\cite{SpickermanJoyner:84} and Dresden and Du \cite{DresdenDu2014}. Levesque \cite{Levesque:85} derived a closed-form expression for sequences of the form \eqref{genfib} in which each term is multiplied with a different factor. Another interesting paper by Wolfram \cite{Wolfram:98} considers explicit formulas for the $k$-generalized Fibonacci sequence with arbitrary starting values, but we will not pursue those ideas here. 

The characteristic polynomial corresponding to the recursion relation \eqref{genfib} is given by
\begin{equation}
\label{polybeta}
p_k(x) = x^k - (\beta-1)\sum_{i=0}^{k-1} x^i.
\end{equation}
The following result concerns properties of the roots of this polynomial. The proof closely follows Miller \cite{Miller:71}. For alternative proofs for the special case $\beta=2$, see \cite{Miles:60, Wolfram:98}.

\begin{lemma}\label{lemma_simple_roots}
Let $k \geq 2$ and $\beta \geq 2$ be integers. Then
\begin{enumerate}
\item[(i)]   the polynomial $p_k$ has a real root $1 < r_{k,1} < \beta$;
\item[(ii)]  the remaining roots $r_{k,2},\dots,r_{k,k}$ of $p_k$ lie within the unit circle of the complex plane;
\item[(iii)] the roots of $p_k$ are simple.
\end{enumerate}
\end{lemma}

\begin{proof}
(i) Descartes' rule of signs implies that $p_k$ has exactly one positive root $r_{k,1}$. Since
\[
p_k(1) = 1-k(\beta-1) < 0 \quad\text{and}\quad p_k(\beta)=1
\]
the Intermediate Value Theorem implies the existence of a root $1 < r_{k,1} < \beta$.

(ii) Define the polynomial
\[
q_k(x) = (x-1) p_k(x) = x^{k+1} - \beta x^k  + \beta-1,
\]
and make the following observations:
\begin{itemize}
\item[(O1)] if $x > r_{k,1}$, then $p_k(x) > 0$, and if $0 < x < r_{k,1}$, then $p_k(x) < 0$;
\item[(O2)] if $x > r_{k,1}$, then $q_k(x) > 0$, and if $1 < x < r_{k,1}$, then $q_k(x) < 0$.
\end{itemize}

Note that $p_k$ has no root $r$ such that $|r| > r_{k,1}$. Indeed, if such a root exists, then $p_k(r)=0$, or, equivalently, $r^k = (\beta-1)\sum_{i=0}^{k-1} r^i$. The triangle inequality gives $|r|^k \leq (\beta-1)\sum_{i=0}^{k-1} |r|^i$. Hence, $p_k(|r|) \leq 0$, which contradicts observation~(O1).

In addition, $p_k$ has no root $r$ with $1 < |r| < r_{k,1}$. Indeed, if such a root exists, then $q_k(r) = (r-1)p_k(r) = 0$ so that $\beta r^k = r^{k+1} + \beta-1$. The triangle inequality implies that $\beta |r|^k \leq |r|^{k+1} + \beta-1$. Hence, $q_k(|r|) \geq 0$, which contradicts observation~(O2).

Finally, $p_k$ has no root $r$ with either $|r|=1$ or $|r| = r_{k,1}$ but $r \neq r_{k,1}$. Indeed, if such a root exists, then $q_k(r)=(r-1)p_k(r)=0$, which implies $\beta r^k = r^{k+1} + \beta-1$ and
\begin{equation}\label{ineq}
\beta|r|^k = |r^{k+1} + \beta-1| \leq |r|^{k+1} + \beta-1.
\end{equation}
If the inequality in \eqref{ineq} is strict, then $q_k(|r|) > 0$. Since $q_k(1) = 0$ and $q_k(r_{k,1})=0$ it then follows that $|r| \neq 1$ and $|r| \neq r_{k,1}$. If the inequality in \eqref{ineq} is an equality, then $r^{k+1}$ must be real. Since $q_k(r)=0$, it follows that $r^k = ((\beta-1)+r^{k+1})/\beta$ is real as well and hence $r$ itself is real. An application of Descartes' rule of signs to $q_k$ implies that when $k$ is even $p_k$ has one negative root, and when $k$ is odd $p_k$ has no negative root. If $k$ is even, then $p_k(0) = -(\beta-1)$ and $p_k(-1) = 1$. By the Intermediate Value Theorem it follows that $-1 < r < 0$. We conclude that no root of $p_k$, except $r_{k,1}$ itself, has absolute value 1 or $r_{k,1}$.

(iii) If $p_k$ has a multiple root, then so has $q_k$. In that case, there exists $r$ such that $q_k(r) = q_k'(r) = 0$. Note that $q_k'(r)=0$ implies that $r=0$ or $r=\beta k/(k+1)$. Clearly, $r=0$ is not a root of $q_k$. By the Rational Root Theorem it follows that the only rational roots of $q_k$ can be integers that divide $\beta-1$. Hence, $r=\beta k/(k+1)$ is not a root of $q_k$ either. We conclude that $q_k$, and thus $p_k$, cannot have multiple roots.
\end{proof}

The proof of the following result closely follows the method of Spickerman and Joyner~\cite{SpickermanJoyner:84} and then uses a rewriting step as in Dresden and Du \cite{DresdenDu2014}.

\begin{lemma}\label{binet}
The sequence $(F_n)$ as defined in \eqref{genfib} is given by the following Binet formula:
\[
F_n = \sum_{j=1}^k \frac{r_{k,j} - 1}{\beta + (k+1)(r_{k,j}-\beta)}r_{k,j}^{n-1},
\]
where $r_{k,1},\dots,r_{k,k}$ are the roots of the polynomial $p_k$ defined in \eqref{polybeta}.
\end{lemma}

\begin{proof}
The generating function of the sequence $(F_n)$ is given by
\[
G(x) = \sum_{n=0}^\infty F_{n+1} x^n.
\]
The equation
\[
\sum_{n=k}^\infty \bigg(F_{n+1} - (\beta-1)\sum_{i=1}^k F_{n+1-i}\bigg)x^n = 0
\]
leads to
\[
G(x) = \sum_{n=0}^{k-1} F_{n+1}x^n - (\beta-1)\sum_{i=1}^{k-1}\sum_{n=0}^{k-i-1} F_{n+1}x^n + (\beta-1)G(x)\sum_{i=1}^k x^i.
\]
Finally, using that $F_1=1$ and $F_n=(\beta-1)\beta^{n-2}$ for $2 \leq n \leq k-1$ implies that
\[
G(x) = \frac{1}{1 - (\beta-1)\sum_{i=1}^k x^i}.
\]

Note that $1/r$ is a root of the denominator of $G$ if and only if $r$ is a root of the characteristic polynomial $p_k$. By Lemma \ref{lemma_simple_roots} part (iii) we can expand the generating function in terms of partial fractions as follows:
\[
G(x) = \sum_{j=1}^k \frac{c_j}{x - 1/r_{k,j}},
\]
where the coefficients are given by
\[
c_j = \lim_{x\to 1/r_{k,j}} \frac{x-1/r_{k,j}}{1 - (\beta-1)\sum_{i=1}^k x^i} = -\frac{1}{(\beta-1)\sum_{i=1}^k i (1/r_{k,j})^{i-1}}.
\]
Observe that
\[
\begin{split}
\bigg(1 - \frac{1}{r_{k,j}}\bigg)\sum_{i=1}^k i\bigg(\frac{1}{r_{k,j}}\bigg)^{i-1}
 & = \sum_{i=1}^k \bigg[ i\bigg(\frac{1}{r_{k,j}}\bigg)^{i-1} - (i+1)\bigg(\frac{1}{r_{k,j}}\bigg)^{i}\bigg] + \sum_{i=1}^k \bigg(\frac{1}{r_{k,j}}\bigg)^{i} \\
 & = 1 - (k+1)\frac{1}{r_{k,j}^k} + \frac{1}{\beta-1}.
\end{split}
\]
This results in
\[
c_j = -\frac{1-1/r_{k,j}}{\beta - (\beta-1)(k+1)/r_{k,j}^k}.
\]
Since $r_{k,j}^{k+1}-\beta r_{k,j}^k + \beta - 1 = (r_{k,j}-1)p(r_{k,j}) = 0$ it follows that $r_{k,j}-\beta = (1-\beta)/r_{k,j}^k$ so that
\[
c_j = -\frac{1-1/r_{k,j}}{\beta + (k+1)(r_{k,j}-\beta)}.
\]
Finally, we have that
\[
G(x) = \sum_{j=1}^k c_j \bigg(-r_{k,j}\sum_{n=0}^\infty r_{k,j}^n x^n \bigg)
=
\sum_{n=0}^\infty \bigg(-\sum_{j=1}^k c_j r_{k,j}^{n+1}\bigg)x^n.
\]
Substituting the values for the coefficients completes the proof.
\end{proof}

For the special case $\beta = 2$ Dresden and Du \cite{DresdenDu2014} go one step further and  derive the following simplified Binet formula:
\[
F_n = \bigg\lfloor\frac{r_{k,1} - 1}{\beta + (k+1)(r_{k,1}-\beta)}r_{k,1}^{n-1} + \frac{1}{2}\bigg\rfloor
\quad\text{for}\quad
n \geq k-2,
\]
where $r_{k,1}$ is the unique root of $p_k$ for which $1 < r_{k,1} < \beta$; see Lemma \ref{lemma_simple_roots}. We expect that this formula can be proven for all integers $\beta>1$ for $n$ sufficiently large, where the lower bound on $n$ may depend on both $\beta$ and $k$. However, we will not pursue this question in this paper.

%===============================================================================

\section{Exponentially growing sequences}

In preparation to the proof of Theorem \ref{main_result} we will prove two facts on sequences that exhibit exponential growth. The first result is a variation on a well-known limit:

\begin{lemma}\label{exponential}
If $(a_k)$ is a sequence such that $\lim_{k\to\infty} k a_k = c$, then
\[
\lim_{k\to\infty} (1-a_k)^k = e^{-c}.
\]
\end{lemma}

\begin{proof}
Let $\varepsilon>0$ be arbitrary. Then there exists $N \in \mathbb{N}$ such that $|k a_k - c| \leq \varepsilon$, or, equivalently,
\[
\bigg(1-\frac{c+\varepsilon}{k}\bigg)^k \leq (1-a_k)^k \leq \bigg(1-\frac{c-\varepsilon}{k}\bigg)^k
\]
for all $k \geq N$. Hence, we obtain
\[
e^{-(c+\varepsilon)} \leq \liminf_{k\to\infty}\, (1-a_k)^k \leq \limsup_{k\to\infty}\, (1-a_k)^k \leq e^{-(c-\varepsilon)}.
\]
Since $\varepsilon>0$ is arbitrary, the result follows.
\end{proof}

The next result provides sufficient conditions under which the difference of two exponentially increasing sequences grows at a linear rate:

\begin{lemma}\label{exponentialdiff}
If $a>1$ and $(b_k)$ is a positive sequence such that $\lim_{k\to\infty} a^k b_k = c$, then
\[
\lim_{k\to\infty} \frac{a^k - (a - b_k)^k}{k} = \frac{c}{a}.
\]
\end{lemma}

\begin{proof}
The algebraic identity
\[
x^k - y^k = (x-y)\sum_{i=0}^{k-1} x^{k-1-i} y^{i}
\]
leads to
\[
\frac{a^k - (a - b_k)^k}{k} =\frac{a^k b_k}{a} \cdot S_k
\quad\text{where}\quad S_k = \frac{1}{k} \sum_{i=0}^{k-1} \bigg(1- \frac{b_k}{a}\bigg)^{i}.
\]
It suffices to show that $\lim_{k\to\infty} S_k = 1$. To that end, note that the assumption implies that $\lim_{k\to\infty} b_k = 0$ so that $-1 < -b_k / a < 0$ for $k$ sufficiently large. Bernoulli's inequality gives
\[
1 - i\frac{b_k}{a} \leq \bigg(1- \frac{b_k}{a}\bigg)^{i} < 1,
\]
which implies that
\[
1 - \frac{k-1}{2}\cdot\frac{b_k}{a} < S_k < 1
\]
for $k$ sufficiently large. Moreover, the assumption implies that $\lim_{k\to\infty} kb_k = 0$. An application of the Squeeze Theorem completes the proof.
\end{proof}

%===============================================================================

\section{Proof of the extreme value law}

Let $\lambda>0$ and define $n_k = \lfloor \beta^k \lambda\rfloor$. Combining Lemma \ref{prob_fib} and \ref{binet} gives
\[
\prob(M_{n_k} \leq 1 - \beta^{-k})
=
\frac{\beta}{\beta-1} \sum_{i=1}^k a_i(k)
\quad\text{where}\quad
a_i(k) = \frac{r_{k,i}-1}{\beta + (k+1)(r_{k,i}-\beta)}\bigg(\frac{r_{k,i}}{\beta}\bigg)^{n_k+k},
\]
where $r_{k,i}$ are the roots of $p_k$. Recall that $r_{k,1}$ is the unique root in the interval $(1,\beta)$, and that $|r_{k,i}|<1$ for $i=2,\dots,k$.
In the remainder of this section Theorem \ref{main_result} will be proven by a careful analysis of the asymptotic behaviour of the dominant root $r_{k,1}$.

We define the following numbers:
\[
r_{k,{\rm min}} = \beta - \frac{\beta-1}{\beta^k-1}(1 + \beta^{-k/2})
\quad\text{and}\quad
r_{k,{\rm max}} = \beta - \frac{\beta-1}{\beta^k-1}.
\]
The number $r_{k,{\rm max}}$ is obtained by applying a single iteration of Newton's method to $p_k$ using the starting point $x=\beta$. The number $r_{k,{\rm min}}$ is a correction of $r_{k,{\rm max}}$ with an exponentially decreasing factor.

\begin{lemma}\label{rootlemma_beta}
If $\beta \geq 2$ is an integer and $k \in \mathbb{N}$ is sufficiently large, then
\begin{enumerate}
\item[(i)]   $p_k(r_{k,{\rm max}}) > 0$;
\item[(ii)]  $p_k(r_{k,{\rm min}}) < 0$;
\item[(iii)] $r_{k,{\rm min}} < r_{k,1} < r_{k,{\rm max}}$.
\end{enumerate}
\end{lemma}

\begin{proof}
(i) For $x \neq 1$ we have
\[
p_k(x) = x^k - (\beta-1)\sum_{i=0}^{k-1} x^i = x^k - (\beta-1)\frac{1-x^k}{1-x} = \frac{1}{1-x}\big((\beta-x)x^k-(\beta-1)\big).
\]
In particular, for $k \geq 2$ it follows that
\[
p_k(r_{k,{\rm max}}) = \frac{1}{\beta^k-2}\bigg[\beta^k-\bigg(\beta - \frac{\beta-1}{\beta^k-1}\bigg)^k - 1\bigg].
\]
It suffices to show that the expression between brackets is positive for $k$ sufficiently large. Lemma \ref{exponentialdiff} gives
\[
\lim_{k\to\infty} \frac{1}{k}\bigg(\beta^k-\bigg(\beta - \frac{\beta-1}{\beta^k-1}\bigg)^k\bigg) = \frac{\beta-1}{\beta}.
\]
Hence, for $k$ sufficiently large it follows that
\[
\beta^k-\bigg(\beta - \frac{\beta-1}{\beta^k-1}\bigg)^k - 1 \geq \frac{\beta-1}{2\beta}k - 1,
\]
and the right-hand side is positive for $k > 2\beta / (\beta-1)$.

(ii) Similar to the proof of part (i) it follows that
\[
p_k(r_{k,{\rm min}}) = \frac{1}{2 + \beta^{-k/2} - \beta^k}\bigg[\bigg(\beta-\frac{\beta-1}{\beta^k-1}(1+\beta^{-k/2})\bigg)^k(1+\beta^{-k/2})-\beta^k+1\bigg].
\]
It suffices to show that the expression between brackets is positive for $k$ sufficiently large. Lemma \ref{exponentialdiff} gives
\[
\lim_{k\to\infty} \frac{1}{k}\bigg(\beta^k - \bigg(\beta-\frac{\beta-1}{\beta^k-1}(1+\beta^{-k/2})\bigg)^k\bigg) = \frac{\beta-1}{\beta}.
\]
Hence, for $k$ sufficiently large it follows that
\[
\beta^k - \bigg(\beta-\frac{\beta-1}{\beta^k-1}(1+\beta^{-k/2})\bigg)^k \leq k.
\]
This gives
\[
\begin{split}
& \bigg(\beta-\frac{\beta-1}{\beta^k-1}(1+\beta^{-k/2})\bigg)^k(1+\beta^{-k/2})-\beta^k+1 \\
 & \hspace{3cm} = \beta^{k/2} + 1 - (1+\beta^{-k/2})\bigg(\beta^k - \bigg(\beta-\frac{\beta-1}{\beta^k-1}(1+\beta^{-k/2})\bigg)^k\bigg) \\
 & \hspace{3cm} \geq \beta^{k/2} + 1 - (1+\beta^{-k/2})k,
\end{split}
\]
and the right-hand side is positive for $k$ sufficiently large.

(iii) By the Intermediate Value Theorem there exists a point $c \in (r_{k,{\rm min}}, r_{k,{\rm max}})$ such that $p_k(c)=0$. Note that $c > 1$ for $k$ sufficiently large. Since $r_{k,1}$ is the only zero of $p_k$ which lies outside the unit circle it follows that $c = r_{k,1}$.
\end{proof}

\noindent
In the particular, for $\beta=2$ the previous result improves the bound $2(1-2^{-k}) < r_{1,k} < 2$ derived by Wolfram \cite{Wolfram:98}.

\begin{lemma}\label{a_limit}
We have that
\[
\lim_{k\to\infty} a_1(k)
=
\frac{\beta-1}{\beta}e^{-\frac{\beta-1}{\beta}\lambda}.
\]
\end{lemma}

\begin{proof}
From Lemma \ref{rootlemma_beta} it follows for sufficiently large $k$ that
\begin{equation}\label{root_bound_beta}
\beta - \frac{\beta-1}{\beta^k-1}(1+\beta^{-k/2}) < r_{k,1} < \beta - \frac{\beta-1}{\beta^k-1}.
\end{equation}
In particular, this implies
\[
\lim_{k\to\infty} r_{k,1} = \beta
\quad\text{and}\quad
\lim_{k\to\infty} (k+1)(r_{k,1}-\beta) = 0
\]
so that
\begin{equation}
\label{limit_unity_beta}
\lim_{k\to\infty} \frac{r_{k,1}-1}{\beta+(k+1)(r_{k,1}-\beta)} = \frac{\beta-1}{\beta}.
\end{equation}
Define the sequences
\[
a_k = \frac{\beta-1}{\beta^{k+1}-\beta}(1+\beta^{-k/2})
\quad\text{and}\quad
b_k = \frac{\beta-1}{\beta^{k+1}-\beta}.
\]
The inequality $\beta^k\lambda-1 \leq n_k \leq \beta^k\lambda$ combined with \eqref{root_bound_beta} implies that
\begin{equation}
\label{final_squeeze_beta}
(1-a_k)^{\beta^k \lambda-1+k} \leq \bigg(\frac{r_{k,1}}{\beta}\bigg)^{n_k+k} \leq (1 - b_k)^{\beta^k\lambda+k}.
\end{equation}
By Lemma \ref{exponential} it follows that
\[
\lim_{k\to\infty} (1 - b_k)^{\beta^{k+1}} = e^{-\frac{\beta-1}{\beta}}
\quad\text{and}\quad
\lim_{k\to\infty} (1 - b_k)^{k} = 1,
\]
which implies that
\[
\lim_{k\to\infty} (1 - b_k)^{\beta^k\lambda+k} = e^{-\frac{\beta-1}{\beta}\lambda}.
\]
A similar result holds for the sequence $(a_k)$. Hence, \eqref{limit_unity_beta} together with the Squeeze Theorem applied to \eqref{final_squeeze_beta} completes the proof.
\end{proof}

\begin{lemma}\label{a_bound}
For $k$ sufficiently large we have that
\[
|a_i(k)| < \frac{2}{|\beta + (k+1)(1-\beta)|}\cdot\frac{1}{\beta^{n_k+k}} \quad\text{for}\quad i=2,\dots,k.
\]
\end{lemma}

\begin{proof}
Using that $|r_{k,i}| < 1$ for $i=2,\dots,k$ gives
\[
|a_i(k)|
 = \frac{|r_{k,i}-1|}{|\beta + (k+1)(r_{k,i}-\beta)|}\cdot\bigg(\frac{|r_{k,i}|}{\beta}\bigg)^{n_k+k}
 < \frac{2}{|\beta + (k+1)(r_{k,i}-\beta)|}\cdot\frac{1}{\beta^{n_k+k}}.
\]
For $z \in \mathbb{C}$ we consider the function
\[
f(z) = \beta + (k+1)(z-\beta).
\]
Writing $z = x+iy$ gives
\[
\begin{split}
|f(z)|^2
 & = (\beta + (k+1)(x-\beta))^2 + (k+1)^2 y^2 \\
 & \geq (\beta + (k+1)(x-\beta))^2.
\end{split}
\]
The quadratic function in the right-hand side attains its minimum value at $x_k = \beta-\beta/(k+1)$, and for $k$ sufficiently large it follows that $x_k > 1$. Using that $\RE(r_{k,i}) \in (-1,1)$ gives
\[
|f(r_{k,i})| \geq |\beta + (k+1)(1-\beta)|.
\]
This completes the proof.
\end{proof}

From Lemma \ref{a_bound} it follows for $k$ sufficiently large that
\[
\bigg| \sum_{i=2}^k a_i(k)\bigg|
 \leq \sum_{i=2}^k |a_i(k)| 
 \leq \frac{2(k-1)}{|\beta + (k+1)(1-\beta)|}\cdot\frac{1}{\beta^{n_k+k}},
\]
so that Lemma \ref{a_limit} implies that
\[
\lim_{k\to\infty} \prob(M_{n_k} \leq 1-\beta^{-k})
=
\lim_{k\to\infty} \frac{\beta}{\beta-1}\sum_{i=1}^k a_i(k) = \lim_{k\to\infty} \frac{\beta}{\beta-1}a_1(k) = e^{-\frac{\beta-1}{\beta}\lambda},
\]
whereby Theorem \ref{main_result} has been proven.

%===============================================================================

\appendix

%===============================================================================

\section{The conditions $D(u_n)$ and $D'(u_n)$}
\label{sec:mixing}

A more general approach to study extreme value laws is to determine for which sequences $(u_n)$, depending on a parameter $\lambda\geq 0$, it follows that
\[
\lim_{n\to\infty} \prob(M_n \leq u_n) = e^{-\lambda}.
\]
In \cite[Theorem 1.5.1]{Leadbetter} the following equivalence is proven: if $(X_i)_{i=0}^\infty$ is an i.i.d.\ sequence of random variables and $\lambda \geq 0$, then
\begin{equation}
\label{eq:equiv_leadbetter}
\lim_{n\to\infty} \prob(M_n \leq u_n) = e^{-\lambda}
\quad\Leftrightarrow\quad
\lim_{n\to\infty} n\prob(X_0 > u_n) = \lambda.
\end{equation}
For example, if the variables $X_i\sim U(0,1)$ are independent, then with $u_n = 1-\lambda/n$ it clearly follows that $n\prob(X_0 > u_n) = \lambda$ for all $n$ and the left hand side of \eqref{eq:equiv_leadbetter} yields precisely the statement in \eqref{uniform_iid}.

When the variables $X_i$ are generated by a dynamical system, and therefore dependent, the equivalence \eqref{eq:equiv_leadbetter} need not hold in general and additional conditions need to be satisfied. Let $f : M \to M$ be a map on a manifold $M$ admitting an invariant Borel probability measure $\mu$. In addition, consider a random variable $X : M \to \mathbb{R}$ on the probability space $(M, \mathcal{B}, \mu)$, where $\mathcal{B}$ is the Borel $\sigma$-algebra on $M$, with $\prob(X \leq u) = \mu(X^{-1}(-\infty, u]))$. The sequence $X_i = X \circ f^i$ is identically distributed but not independent. Based on \cite{Leadbetter} the following two conditions were presented in \cite{FF:08b}:

\begin{definition}\label{def:dun}
The condition $D(u_n)$ holds for the sequences $(X_i)_{i=0}^\infty$ and
$(u_n)_{n=1}^\infty$ if for any integers $\ell,t,n\geq 1$ we have
\[
\begin{split}
\big|\prob (X_0 > u_n, & X_t \leq u_n,\dots,X_{t+\ell-1} \leq u_n)  \\
&  - \prob (X_0 > u_n) \prob(X_t \leq u_n,\dots,X_{t+\ell-1} \leq u_n)\big| \leq \gamma(n,t),
\end{split}
\]
where $\gamma (n,t)$ is non-increasing in $t$ for each $n$ and $n\gamma(n,t_n)\to 0$
as $n\to \infty$ for some sequence $t_n = o(n)$ as $t_n\rightarrow \infty$.
\end{definition}

\begin{definition}\label{def:dprime}
The condition $D'(u_n)$ holds for the sequences $(X_i)_{i=0}^\infty$ and
$(u_n)_{n=1}^\infty$ if 
\[
\lim_{k\to \infty}\bigg(\limsup_{n\to\infty} n\sum_{j=1}^{\lfloor n/k \rfloor}\prob(X_0>u_n,X_j>u_n)\bigg) = 0.
\]
\end{definition}

The $D(u_n)$ condition imposes a decay rate on the dependence of specific events concerning threshold exceedances. The $D'(u_n)$ condition restricts the amount of clustering of exceedances over a threshold. Under these two conditions the equivalence in \eqref{eq:equiv_leadbetter} remains true for the process $X_i = X \circ f^i$ \cite[Theorem 1]{FF:08b}. For the R\'enyi map process we will now show that $D(u_n)$ is satisfied, but $D'(u_n)$ is not.

It follows from \cite[Theorem 8.3.2]{BG:97} that the R\'enyi map has exponential decay of correlations. This means the following: for all functions $\varphi\in BV([0,1))$ and $\psi \in L^{\infty}([0,1))$ there exist constants $C>0$ and $0 < r < 1$ such that
\begin{equation*}\label{eq:decay}
\bigg| \int_0^1 \varphi \cdot (\psi \circ f^t)d\mu - \int_0^1 \varphi d\mu \int_0^1 \psi d\mu\bigg|
\leq
C \text{Var}(\varphi)\|\psi\|_\infty r^t \quad \text{for all}\quad t \geq 0.
\end{equation*}
By taking the indicator functions $\varphi = 1_{\{X_0>u_n\}}$ and $\psi = 1_{\{X_0 \leq u_n, \dots, X_{\ell-1} \leq u_n\}}$  it follows that the $D(u_n)$ condition is satisfied with $\gamma(n,t) = 2Cr^t$ and $t_n = n^\alpha$ for any $0 < \alpha < 1$.

Now we show that $D'(u_n)$ does not hold for any sequence $u_n$ that satisfies
\[
\lim_{n\to\infty} n\prob(X_0>u_n) = \lim_{n\to\infty} n(1-u_n) = \lambda > 0.
\]
To that end, observe that we have the following inclusion:
\[
(u_n,1) \cap f^{-j}((u_n,1)) \supset \bigg(1-\frac{1-u_n}{\beta^j}, 1\bigg).
\]
This gives the inequality
\[
\prob(X_0 > u_n, X_j > u_n) = \Leb\big((u_n,1) \cap f^{-j}((u_n,1))\big) \geq \Leb\bigg(1-\frac{1-u_n}{\beta^j}, 1\bigg) = \frac{1-u_n}{\beta^j},
\]
which implies that
\[
n \sum_{j=1}^{\lfloor n/k\rfloor} \prob(X_0 > u_n, X_j > u_n)
  \geq n(1-u_n) \sum_{j=1}^{\lfloor n/k\rfloor} \frac{1}{\beta^j}
  = n(1-u_n)\cdot\frac{1-\beta^{-\lfloor n/k\rfloor}}{\beta-1}.
\]
Finally, it follows that
\[
\lim_{k\to\infty} \bigg(\limsup_{n\to\infty} n \sum_{j=1}^{\lfloor n/k\rfloor} \prob(X_0 > u_n, X_j > u_n)\bigg)
\geq \frac{\lambda}{\beta-1} > 0,
\]
which shows that the $D'(u_n)$ condition is not satisfied.

%===============================================================================

\section{Clustering and the extremal index}
\label{sec:clustering}

Extremes in the R\'enyi map process can form clusters. Let $u = \beta^{-k}$ for some $k \in \mathbb{N}$. The probability of having a cluster of $q$ consecutive variables $X_i$ exceeding the threshold $1-u$ is given by
\[
\frac{\prob(X_0,\dots,X_{q-1}>1-u, X_q \leq 1-u)}{\prob(X_0 > 1-u)}
=
\frac{\Leb(E_0 \cap \dots \cap E_{q-1} \cap E_q^c)}{\Leb(E_0)}.
\]
Observe that
\[
E_0 \cap \dots \cap E_{q-1} = \bigg[\frac{\beta^{q-1}-u}{\beta^{q-1}}, 1\bigg)
\quad\text{and}\quad
E_q^c = \bigcup_{j=1}^{\beta^q} \bigg[\frac{j-1}{\beta^q},\frac{j-u}{\beta^q}\bigg),
\]
which implies that
\[
E_0 \cap \dots \cap E_{q-1} \cap E_q^c = \bigg[\frac{\beta^{q-1}-u}{\beta^{q-1}},\frac{\beta^q-u}{\beta^q}\bigg).
\]
Hence, the probability of the occurrence of a cluster of length $q$ is given by
\[
\frac{\prob(X_0,\dots,X_{q-1}>1-u, X_q \leq 1-u)}{\prob(X_0 > 1-u)} = \frac{\beta-1}{\beta}\cdot\frac{1}{\beta^{q-1}}.
\]
Using that $1 + 2x + 3x^2 + \dots = 1/(1-x)^2$ for $|x|<1$ implies that the mean cluster size is given by
\[
\mean(\text{cluster size})
 = \sum_{q=1}^\infty q\prob(\text{cluster of size } q)
 = \frac{\beta-1}{\beta}\sum_{q=1}^\infty \frac{q}{\beta^{q-1}}
 = \frac{\beta-1}{\beta}\bigg(\frac{\beta}{\beta-1}\bigg)^2
 = \frac{\beta}{\beta-1}.
\]
Finally, by taking the reciprocal of the mean cluster size we obtain the extremal index $\theta = (\beta-1)/\beta$. It is precisely the clustering of extremes which violates the $D'(u_n)$ condition that was discussed in the previous section.

%===============================================================================

%\bibliographystyle{plain}
%\bibliography{doubling-map-refs}

%===============================================================================

\end{document}